\documentclass[11pt]{amsart}
\usepackage{amssymb,amsmath,amsfonts,amscd,euscript}
\usepackage[all]{xy}

\newcommand{\nc}{\newcommand}

\numberwithin{equation}{section}
\newtheorem{thm}{Theorem}[section]
\newtheorem{prop}[thm]{Proposition}
\newtheorem{lem}[thm]{Lemma}
\newtheorem{cor}[thm]{Corollary}
\theoremstyle{remark}
\newtheorem{rem}[thm]{Remark}

\newtheorem{dfn}[thm]{Definition}

\nc{\gl}{\mathfrak{gl}}
\nc{\GL}{\mathfrak{GL}}
\nc{\g}{\mathfrak{g}}
\nc{\gh}{\widehat\g}
\nc{\h}{\mathfrak{h}}
\nc{\la}{\lambda}
\nc{\al}{\alpha }
\nc{\be}{\beta }
\nc{\ve}{\varepsilon }
\nc{\om}{\omega }

\nc{\ta}{\theta}
\nc{\veps}{\varepsilon}
\nc{\ch}{{\mathop {\rm ch}}}
\nc{\Tr}{{\mathop {\rm Tr}\,}}
\nc{\Id}{{\mathop {\rm Id}}}
\nc{\ad}{{\mathop {\rm ad}}}
\nc{\bra}{\langle}
\nc{\ket}{\rangle}
\nc{\x}{{\bf x}}
\nc{\bs}{{\bf s}}
\nc{\bp}{{\bf p}}
\nc{\bc}{{\bf c}}
\nc{\pa}{\partial}
\nc{\ld}{\ldots}
\nc{\cd}{\cdots}
\nc{\hk}{\hookrightarrow}
\nc{\T}{\otimes}
\newcommand{\bea}{\begin{equation}}
\newcommand{\ena}{\end{equation}}
\nc{\gr}{\mathrm{gr}}
\nc{\ov}{\overline}

\nc{\cO}{\mathcal O}
\nc{\msl}{\mathfrak{sl}}
\nc{\mgl}{\mathfrak{gl}}
\nc{\U}{\mathrm U}
\nc{\V}{\EuScript V}
\nc{\bH}{\EuScript H}
\nc{\Res}{\mathrm{Res\ }}

\newcommand{\bC}{{\mathbb C}}

\newcommand{\Fl}{\EuScript{F}}

\newcommand{\spa}{\mathrm{span}}
\newcommand{\bff}{{\bf f}}
\newcommand{\qbinom}[2]{\binom{#1}{#2}_{\!\! q}}

\begin{document}

\title[The median Genocchi numbers]
{The median Genocchi numbers, $Q$-analogues and continued fractions}

\author{Evgeny Feigin}
\address{Evgeny Feigin:\newline
Department of Mathematics,\newline National Research University Hisger School of Economics,\newline
Vavilova str. 7, 117312, Moscow, Russia\newline
{\it and }\newline
Tamm Theory Division, Lebedev Physics Institute
}
\email{evgfeig@gmail.com}

\begin{abstract}
The goal of this paper is twofold. First, we review the recently developed geometric approach to the
combinatorics of the median Genocchi numbers. The Genocchi numbers appear in this context as Euler characteristics of the degenerate flag varieties.  
Second, we prove that the generating function of the Poincar\' e polynomials of the degenerate flag varieties
can be written as a simple continued fraction. 
As an application we prove that
the Poincar\' e polynomials coincide with the $q$-version of the normalized median Genocchi
numbers introduced by Han and Zeng.
\end{abstract}

\maketitle

\section*{Introduction}
The Genocchi numbers appear in many different contexts (see e.g. \cite{B}, \cite{Du}, \cite{DR}, \cite{DZ},
\cite{DV}, \cite{G}, \cite{V2}).  
Probably, the most well-known definition uses the Seidel triangle
$$
\begin{matrix}
& & & &  & & & & 155 & 155\\
& & & &  & & 17 & 17 & 155 & 310\\
& & & & 3  & 3 & 17 & 34 & 138 & 448\\
& & 1 & 1 & 3  & 6 & 14 & 48 & 104 & 552\\
1 & 1 & 1 & 2 & 2 & 8 & 8 & 56 & 56 & 608
\end{matrix}
$$
By definition, the triangle is formed by the numbers $g_{k,n}$
($k$ is the number of a row counted from bottom to top and $n$ is the number of a column from left to right)
with $n=1,2,\dots$ and $1\le k\le \frac{n+1}{2}$,
subject to the relations $g_{1,1}=1$ and
\[
g_{k,2n}=\sum_{i\ge k} g_{i,2n-1},\ g_{k,2n+1}=\sum_{i\le k} g_{i,2n}.
\]
For example, 
$138=56+48+34$ and $48=14+17+17$. The two sequences of numbers sitting on the edges of the Seidel triangle
are called the Genocchi numbers. More precisely, the Genocchi numbers of the first kind are
$1,1,3,17,155, \dots$ and of the second kind are $1,2,8,56,608,\dots$. 
The latter numbers are also referred to as
the median Genocchi numbers and are denoted by $H_{2n-1}$. For example, $H_1=1$ and $H_7=56$. These numbers
are known to be divisible by the powers of $2$ (see \cite{B}, \cite{Du}):  $H_{2n+1}\div 2^n$. The ratios are 
called the normalized median Genocchi numbers and are denoted by $h_n$. Thus the first values $h_0,h_1,h_2,\dots$
are as follows: \[1,1,2,7,38,295,3098,\dots\]

It has been shown recently (see \cite{Fe2}) that the numbers $h_n$ are analogues ("degenerations") 
of the numbers $n!$. More precisely, let $\Fl_n$ be the variety of flags in an $n$-dimensional space, i.e. 
$\Fl_n$ consists of collections of subspaces $(V_1\subset V_2\subset \dots\subset V_{n-1})$ of a given
$n$-dimensional space $W$ such that $\dim V_k=k$. It is well known that the Euler characteristics of 
$\Fl_n$ is equal to $n!$. Combinatorially, the number $n!$ appears in this context as the number of sequences 
$(I_1\subset I_2\subset \dots\subset I_{n-1})$ of subsets of $\{1,\dots,n\}$ such that $\# I_k=k$. 
The varieties $\Fl_n$ have natural degenerations $\Fl^a_n$, called the degenerate flag varieties
(see \cite{Fe1},\cite{Fe2},\cite{FF}, \cite{FFL}). 
In order to define $\Fl^a_n$ we fix a basis $w_1,\dots,w_n$ of $W$
and the projection maps $pr_k: W\to W$ mapping $w_k$ to $0$ and $w_i$ with $i\ne k$ to $w_i$. 
The degenerate flag varieties consist of collections $(V_1,V_2,\dots,V_{n-1})$ of subspaces of $W$ such that
$\dim V_k=k$ and $pr_{k+1} V_k\subset V_{k+1}$. It turns out that the Euler characteristic of 
$\Fl^a_n$ is equal to the normalized median Genocchi number:
\begin{equation}\label{Ec}
\chi(\Fl^a_n)=h_n.
\end{equation}
Combinatorially this means that the number of sequences $(I_1,I_2,\dots,I_{n-1})$
of subsets of $\{1,\dots,n\}$ such that $\# I_k=k$ and $I_k\subset I_{k+1}\cup \{k+1\}$ is equal
to $h_n$.

In this paper we review (following \cite{Fe2} and \cite{CFR}) the applications of the observation \eqref{Ec} 
to the combinatorics of $h_n$. We give several new combinatorial objects counted by the
normalized median Genocchi numbers. As an application  the formula for the numbers $h_n$ is derived in
terms of binomial coefficients. Using \eqref{Ec} we introduce natural $q$-analogues $h_n(q)$ as 
Poincar\'e polynomials of the degenerate flag varieties. We note that the degenerate flag varieties are
singular, but share the following important property with their classical analogues: the varieties 
$\Fl^a_n$ can be decomposed into a disjoint union of complex (even-dimensional real) affine cells. 
Therefore the Poincar\'e polynomials $P_{\Fl^a_n}(t)$ are functions of $q=t^2$ (odd powers do
not show up). Hence we define 
\[
h_n(q)=P_{\Fl^a_n}(q^{1/2}). 
\]
Obviously, one has $h_n(1)=h_n$.
We note also that the degree of $h_n(q)$ is equal to $n(n-1)/2$, since the complex dimension of $\Fl^a_n$
is $n(n-1)/2$.
In the paper we recall two formulas for the polynomials $h_n(q)$: one uses certain statistic on the set of 
Dellac configurations
(see \cite{De},\cite{Fe2}) and another is obtained via the geometric arguments (see \cite{CFR}). 
Various $q$-analogues of the Genocchi numbers appear in the literature (see e.g. \cite{HZ1}, \cite{HZ2},
\cite{ZZ}). In particular, in \cite{HZ1} Han and Zeng used the $q$-analogues to give a third proof of the 
Barsky theorem (\cite{B}, \cite{Du}).
 
Our new result is the continued fraction presentation of the generating function of the polynomials $h_n(q)$.
Namely, it is convenient to introduce the "reversed" polynomials $\tilde h_n(q)=q^{n(n-1)/2} h_n(q^{-1})$.
Then we have
\begin{thm}
\[
\sum_{n\ge 0} \tilde h_n(q) s^n = \cfrac{1}{1
          - \cfrac{s}{1
          - \cfrac{qs}{1
          - \cfrac{\qbinom{3}{2}s}{1
          -\cfrac{q\qbinom{3}{2}s}{1
          - \cfrac{\qbinom{4}{2}s}{1
          -\cfrac{q\qbinom{4}{2}s}{1
          -\dots}}}}}}}
\]
\end{thm}
Using this formula, we prove that $\tilde h_n(q)$ coincide with the $q$ version of the normalized median
Genocchi numbers introduced by Han and Zeng in \cite{HZ1}, \cite{HZ2}.
We also show that the Viennot formula (see \cite{V1}, \cite{V3}, \cite{Du}, \cite{DZ}) for the generating 
function of the median Genocchi numbers $H_{2n-1}$ can be derived by specialization at $q=1$.       

Our paper is organized as follows.\\
In Section $1$ we give several definitions of the normalized median Genocchi numbers.\\
In Section $2$ we give two formulas for the polynomials $h_n(q)$.\\
In Section $3$ we obtain the continued fraction presentation for the generating function of 
$\tilde h_n(q)$.

\section{Combinatorics of the normalized median Genocchi numbers}
The normalized median Genocchi numbers $h_n$, $n=0,1,2,\dots$ form a sequence which starts with
$1,1,2,7,38,295$. These numbers enjoy many definitions (see \cite{B}, \cite{Du}, \cite{De}, \cite{G},
\cite{K}, \cite{Sl}, \cite{Fe2}). We recall some of them now.   

\subsection{The Seidel triangle}\label{St}
The Seidel triangle \cite{Se} is formed by the numbers $g_{k,n}$
with $n=1,2,\dots$ and $1\le k\le \frac{n+1}{2}$,
subject to the relations $g_{1,1}=1$ and
\[
g_{k,2n}=\sum_{i\ge k} g_{i,2n-1},\ g_{k,2n+1}=\sum_{i\le k} g_{i,2n}.
\]
The numbers $g_{n,2n-1}$ are called the Genocchi numbers of the first kind and the numbers
$H_{2n-1}=g_{1,2n}$ are called the Genocchi numbers of the second kind (or the median Genocchi numbers).
Barsky \cite{B} and then Dumont \cite{Du} proved that the number $H_{2n+1}$ is divisible by
$2^n$. The normalized median Genocchi numbers $h_n$ are defined as the corresponding ratios:
$h_{n}=H_{2n+1}/2^{n}$.

\subsection{Dellac's configurations}\label{Dc}
The earliest definition was given by Dellac in \cite{De}. Consider a rectangle with $n$ columns and
$2n$ rows. It contains $n\times 2n$ boxes labeled by pairs $(l,j)$, where $l=1,\dots,n$ is the number
of a column and $j=1,\dots,2n$ is the number of  a row. A Dellac  configuration $D$ is a subset of boxes,
subject to the following conditions:
\begin{itemize}
\item each column contains exactly two boxes from $D$,
\item each row contains exactly one box from $D$,
\item if the $(l,j)$-th box is in $D$, then $l\le j\le n+l$.
\end{itemize}
Let $DC_n$ be the set of such configurations. Then the number of elements in $DC_n$ is equal to $h_n$.

We list all Dellac's configurations for $n=3$.
We specify boxes in a configuration by putting  fat dots inside.

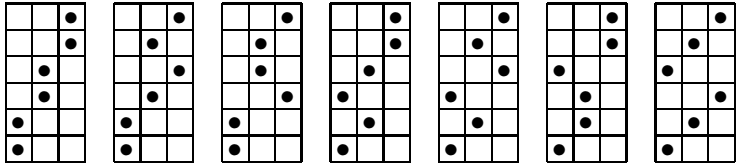
\begin{equation*}%\label{n=3}
\begin{picture}(30,60)
\put(0,0){\line(1,0){30}}
\put(0,10){\line(1,0){30}}
\put(0,20){\line(1,0){30}}
\put(0,30){\line(1,0){30}}
\put(0,40){\line(1,0){30}}
\put(0,50){\line(1,0){30}}
\put(0,60){\line(1,0){30}}

\put(0,0){\line(0,1){60}}
\put(10,0){\line(0,1){60}}
\put(20,0){\line(0,1){60}}
\put(30,0){\line(0,1){60}}

\put(2,2){$\bullet$}
\put(2,12){$\bullet$}
\put(12,22){$\bullet$}
\put(12,32){$\bullet$}
\put(22,42){$\bullet$}
\put(22,52){$\bullet$}
\end{picture} \quad
\begin{picture}(30,60)
\put(0,0){\line(1,0){30}}
\put(0,10){\line(1,0){30}}
\put(0,20){\line(1,0){30}}
\put(0,30){\line(1,0){30}}
\put(0,40){\line(1,0){30}}
\put(0,50){\line(1,0){30}}
\put(0,60){\line(1,0){30}}

\put(0,0){\line(0,1){60}}
\put(10,0){\line(0,1){60}}
\put(20,0){\line(0,1){60}}
\put(30,0){\line(0,1){60}}

\put(2,2){$\bullet$}
\put(2,12){$\bullet$}
\put(12,22){$\bullet$}
\put(12,42){$\bullet$}
\put(22,32){$\bullet$}
\put(22,52){$\bullet$}
\end{picture} \quad
\begin{picture}(30,60)
\put(0,0){\line(1,0){30}}
\put(0,10){\line(1,0){30}}
\put(0,20){\line(1,0){30}}
\put(0,30){\line(1,0){30}}
\put(0,40){\line(1,0){30}}
\put(0,50){\line(1,0){30}}
\put(0,60){\line(1,0){30}}

\put(0,0){\line(0,1){60}}
\put(10,0){\line(0,1){60}}
\put(20,0){\line(0,1){60}}
\put(30,0){\line(0,1){60}}

\put(2,2){$\bullet$}
\put(2,12){$\bullet$}
\put(12,32){$\bullet$}
\put(12,42){$\bullet$}
\put(22,22){$\bullet$}
\put(22,52){$\bullet$}
\end{picture} \quad
\begin{picture}(30,60)
\put(0,0){\line(1,0){30}}
\put(0,10){\line(1,0){30}}
\put(0,20){\line(1,0){30}}
\put(0,30){\line(1,0){30}}
\put(0,40){\line(1,0){30}}
\put(0,50){\line(1,0){30}}
\put(0,60){\line(1,0){30}}

\put(0,0){\line(0,1){60}}
\put(10,0){\line(0,1){60}}
\put(20,0){\line(0,1){60}}
\put(30,0){\line(0,1){60}}

\put(2,2){$\bullet$}
\put(2,22){$\bullet$}
\put(12,12){$\bullet$}
\put(12,32){$\bullet$}
\put(22,42){$\bullet$}
\put(22,52){$\bullet$}
\end{picture}\quad
\begin{picture}(30,60)
\put(0,0){\line(1,0){30}}
\put(0,10){\line(1,0){30}}
\put(0,20){\line(1,0){30}}
\put(0,30){\line(1,0){30}}
\put(0,40){\line(1,0){30}}
\put(0,50){\line(1,0){30}}
\put(0,60){\line(1,0){30}}

\put(0,0){\line(0,1){60}}
\put(10,0){\line(0,1){60}}
\put(20,0){\line(0,1){60}}
\put(30,0){\line(0,1){60}}

\put(2,2){$\bullet$}
\put(2,22){$\bullet$}
\put(12,12){$\bullet$}
\put(12,42){$\bullet$}
\put(22,32){$\bullet$}
\put(22,52){$\bullet$}
\end{picture}\quad
\begin{picture}(30,60)
\put(0,0){\line(1,0){30}}
\put(0,10){\line(1,0){30}}
\put(0,20){\line(1,0){30}}
\put(0,30){\line(1,0){30}}
\put(0,40){\line(1,0){30}}
\put(0,50){\line(1,0){30}}
\put(0,60){\line(1,0){30}}

\put(0,0){\line(0,1){60}}
\put(10,0){\line(0,1){60}}
\put(20,0){\line(0,1){60}}
\put(30,0){\line(0,1){60}}

\put(2,2){$\bullet$}
\put(2,32){$\bullet$}
\put(12,12){$\bullet$}
\put(12,22){$\bullet$}
\put(22,42){$\bullet$}
\put(22,52){$\bullet$}
\end{picture}\quad
\begin{picture}(30,60)
\put(0,0){\line(1,0){30}}
\put(0,10){\line(1,0){30}}
\put(0,20){\line(1,0){30}}
\put(0,30){\line(1,0){30}}
\put(0,40){\line(1,0){30}}
\put(0,50){\line(1,0){30}}
\put(0,60){\line(1,0){30}}

\put(0,0){\line(0,1){60}}
\put(10,0){\line(0,1){60}}
\put(20,0){\line(0,1){60}}
\put(30,0){\line(0,1){60}}

\put(2,2){$\bullet$}
\put(2,32){$\bullet$}
\put(12,12){$\bullet$}
\put(12,42){$\bullet$}
\put(22,22){$\bullet$}
\put(22,52){$\bullet$}
\end{picture}\ .
\end{equation*}

\subsection{Permutations}
In \cite{K} Kreweras suggested another description of the numbers $h_n$. Namely, a permutation
$\sigma\in S_{2n+2}$ is called a normalized Dumont permutation of the second kind if the following conditions
are satisfied:
\begin{itemize}
 \item $\sigma(k)<k$ if $k$ is even,
\item $\sigma(k)>k$ if $k$ is odd,
\item $\sigma^{-1}(2k)<\sigma^{-1}(2k+1)$ for $k=1,\dots,n$.
\end{itemize}
 According to Kreweras, the number of such permutations is equal to $h_n$ 
(see also \cite{Fe2},  Proposition 3.3).

\subsection{A la $n!$}
Let ${\bf I}=(I_1,I_2,\dots,I_{n-1})$ be a sequence of subsets of the set $\{1,\dots,n\}$.
We cal such a sequence admissible if $\# I_l=l$ for all $l$ and
\begin{equation}\label{l+1}
I_l\subset I_{l+1}\cup \{l+1\},\ l=1,\dots,n-2.
\end{equation}
Then the number of admissible sequences is equal to the normalized median Genocchi number $h_n$.   

\begin{rem}
If we replace \eqref{l+1} with $I_l\subset I_{l+1}$, then the number of admissible sequences will obviously
become $n!$. 
\end{rem}

Admissible sequences can be visualized as follows. 
Consider an oriented graph $\Gamma$ with the set of vertices
$(l,j)$,  $1\le l\le n-1$, $1\le j\le n$. Two vertices $(l_1,j_1)$ and $(l_2,j_2)$ are connected by an arrow
$(l_1,j_1)\to (l_2,j_2)$ if and only if $j_1=j_2$, $l_2=l_1+1$, $l_2\ne j_2$ (note that in \cite{CFR} this graph
appeared as a coefficient quiver). 
For $n=5$ the graph is as follows:
\[
\xymatrix@R=6pt@C=8pt
{
(1,5)\ar[r]&(2,5)\ar[r]&(3,5)\ar[r]&(4,5)&\\
(1,4)\ar[r]&(2,4)\ar[r]&(3,4)&(4,4)&\\
(1,3)\ar[r]&(2,3)&(3,3)\ar[r]&(4,3)&\\
(1,2)&(2,2)\ar[r]&(3,2)\ar[r]&(4,2)&\\
(1,1)\ar[r]&(2,1)\ar[r]&(3,1)\ar[r]&(4,1)&\\
}
\] 
To a collection ${\bf I}=(I_1,\dots,I_{n-1})$ with $\#I_l=l$ for all $l$ we associate a subset $S_{\bf I}$ of
vertices of $\Gamma$ by the formula $$S_{\bf I}=\{(l,j):\ j\in I_l\}.$$
Then ${\bf I}$ is admissible if and only if $S_{\bf I}$ is closed in $\Gamma$, i.e. if $p\in S_{\bf I}$
and $p\to q$ is an arrow in $\Gamma$ then $q\in S_{\bf I}$.

\subsection{Euler characteristic.} 
The admissible sequences label the cells in the degenerate flag varieties $\Fl^a_n$ (see \cite{Fe2}). 
Recall that
$\Fl^a_n$ consists of sequences $(V_1,\dots,V_{n-1})$ of the subspaces of a given $n$-dimensional 
space $W$ subject to the conditions given below. Let $w_1,\dots,w_n$ be a basis of $W$ and let 
$pr_k:W\to W$ be the projection along $w_k$ to $\spa(w_i: i\ne k)$: 
$pr_k (\sum_{i=1}^n c_i w_i)=\sum_{i\ne k} c_i w_i$. 
\begin{dfn}
A collection $(V_1,\dots,V_{n-1})$ of subspaces $V_k\subset W$ belongs to $\Fl^a_n$ if and only if
$\dim V_k=k$ for all $k$ and
\[
pr_{k+1} V_k\subset V_{k+1},\ k=1,\dots,n-2.
\]
\end{dfn}
We recall (see e.g. \cite{Fu}) that the classical flag varieties $\Fl_n$ consist of collections 
$(V_1\subset\dots\subset V_{n-1})$ of sequentially embedded subspaces of $W$ with $\dim V_k=k$ for all $k$.
These varieties are acted upon by a torus $T=(\bC^*)^n$. The action is induced from the
natural action of $T$ on $W$: 
\[(a_1,\dots,a_n)\cdot (c_1w_1+\dots +c_nw_n)=a_1c_1w_1+\dots +a_nc_nw_n.\] 
The varieties $\Fl_n$ enjoy several important properties:
\begin{enumerate}
\item \label{i}
$\Fl_n$ can be decomposed into a disjoint union of complex (even-dimensional real) cells. Each cell
contains exactly one $T$-fixed point.
\item \label{ii}
The $T$-fixed points on $\Fl_n$ are labeled by permutations from $S_n$. The fixed  point $p(\sigma)$ attached to
$\sigma\in S_n$ is given by
\[
p(\sigma)=(\spa(w_{\sigma(1)}), \spa(w_{\sigma(1)},w_{\sigma(2)}),\dots, \spa(w_{\sigma(1)},\dots,w_{\sigma(n-1)})). 
\]
\item \label{iii}
The (real) dimension of the cell containing $p(\sigma)$ is equal to $2l(\sigma)$ 
(twice length of the permutation). 
\end{enumerate}
In particular, the Euler characteristic of $\Fl_n$ is equal to $n!$. 
The degenerate flag varieties $\Fl^a_n$ share property \eqref{i}.   
However, the labeling set for the $T$-fixed points is different: the number of torus fixed points 
in $\Fl^a_n$ is equal to the number of admissible
sequences. Namely, an admissible sequence ${\bf I}$ defines a point 
\begin{equation}\label{pI}
p({\bf I})=(V_1,\dots,V_{n-1}),\ V_k=\spa(w_i: i\in I_k).
\end{equation}  
Any such point belongs to $\Fl^a_n$, is $T$-fixed and all $T$-fixed points in $\Fl^a_n$ are of this form.
\begin{cor}
The Euler characteristic of $\Fl^a_n$ is equal to $h_n$:
\[
\chi(\Fl^a_n)=h_n.
\]
\end{cor}

\subsection{Two triangles}
We add one more combinatorial description of the numbers $h_n$ obtained in \cite{CFR} using the representation
theory of quivers.
\begin{prop}
The normalized median Genocchi number $h_{n+1}$ is equal to the number of pairs of collections of non-negative
integers $(r_{i,j})$, $(m_{i,j})$, $1\le i\le j\le n$ subject to the following conditions
for all $k=1,\dots,n$:
$$
\sum_{i=k}^n r_{k,i}\le 1, \quad \sum_{j=1}^k m_{j,k}\le 1,\quad
\sum_{i\le k\le j} r_{i,j}=\sum_{i\le k\le j} m_{i,j}.
$$
\end{prop}

\subsection{Explicit formula}
The  following explicit formula for the numbers $h_n$ is available (see \cite{CFR}) 
\begin{equation}
h_{n}=\sum_{f_0,\dots,f_n\ge 0} \prod_{k=1}^{n-1} \binom{1+f_{k-1}}{f_k}
\prod_{k=1}^{n} \binom{1+f_{k+1}}{f_k}
\end{equation}
with $f_0=f_n=0$.
This formula can be rewritten as a weighted sum of Motzkin paths. Namely,
let $M_n$ be the set of Motzkin paths starting
at $(0,0)$ and ending at $(n,0)$.
For a Motzkin path $\bff=(f_0,\dots,f_n)\in M_{n}$ with $f_0=f_n=0$ let $l(\bff)$
be the number of "rises" ($f_{i+1}=f_i+1$) plus the number of "falls" ($f_{i+1}=f_i-1$). Then
we obtain
\begin{equation}\label{Mp}
h_n=\sum_{\bff\in M_n} \frac{\prod_{k=0}^{n} (1+f_k)^2}{2^{l(\bff)}}.
\end{equation}

\subsection{The Viennot formula.}
In section \ref{cf} we use the $q$-version of the formula \eqref{Mp} to derive the continued fraction presentation 
of the generating function of $q$-Genocchi numbers. We close this section with the continued fraction form
of the generating function of the (non-normalized) median Genocchi numbers $H_{2n-1}$ 
(see subsection \ref{St}) due to Viennot  (\cite{V1}, \cite{V3}, \cite{Du}, \cite{DZ}):
\[
1+\sum_{n=1}^\infty H_{2n-1}x^n=
\cfrac{1}{1-\cfrac{1^2x}{1-\cfrac{1^2x}{1-\cfrac{2^2x}{1-\cfrac{2^2x}{\dots}}}}}
\]

\section{q-versions}
Several $q$-versions of the median Genocchi numbers can be found in the literature 
(see \cite{HZ1}, \cite{HZ1}, \cite{ZZ}). 
We briefly recall the definition of the Han-Zeng polynomials below.
In our approach the normalized median Genocchi numbers appear as the Euler 
characteristics of the degenerate flag varieties. Thus we obtain a natural $q$-analogue defined by the Poincar\'e
polynomials of $\Fl^a_n$. We give combinatorial description as well as an explicit formula for these polynomials
below. 

\subsection{The Han-Zeng polynomials}
Consider the polynomials $C_n(x,q)$ in two variables defined by $C_1(x,q)=1$ and
\[
C_n(x,q)=(1 + qx)\frac{(1 + qx)C_{n-1}(1 + qx,q) - xC_{n-1}(x,q)}{1+qx-x},\ n\ge 2.
\]
Define 
\[
\bar c_n(q) = \frac{C_n(1,q)}{(1 + q)^{n-1}},\ n\ge 1,   
\]
Han and Zeng proved that these are polynomials satisfying $\bar c_n(1)=h_{n-1}$ 
(see formula $(17)$ in \cite{HZ1}). Hence $\bar c_n(q)$ can be viewed as 
$q$-analogues of the normalized median Genocchi numbers.

\subsection{Statistic on the Dellac configurations.}
For a Dellac configuration
$D\in DC_n$ (see subsection \ref{Dc}) 
we define the length $l(D)$ of $D$ as the number of pairs $(l_1,j_1)$, $(l_2,j_2)$ such that
the boxes $(l_1,j_1)$ and $(l_2,j_2)$ are both in $D$ and $l_1<l_2$, $j_1>j_2$.
%We call such a pair of boxes $(l_1,j_1)$,  $(l_2,j_2)$ a disorder.
This definition resembles the definition of the length of a permutation.
We note that in the classical case the complex dimension of the cell attached to  a permutation $\sigma$ in a
flag variety is equal to the number of pairs $j_1<j_2$ such that $\sigma (j_1)>\sigma(j_2)$,
which equals to the length of $\sigma$ (see property \eqref{iii} of $\Fl_n$).
\begin{prop}
The real dimension of the cell in $\Fl^a_n$ containing a point $p({\bf I})$ \eqref{pI} is equal to $2l(D)$. 
Thus the Poincar\' e polynomial $h_n(t)=P_{\Fl^a_n}(t)$ is given by
$h_n(t)=\sum_{D\in DC_n} t^{2l(D)}.$
\end{prop}
Since the polynomials $h_n(t)$ do not contain odd powers of $t$, it is convenient to introduce a new variable
$q=t^2$ and define $h_n(q)=P_{\Fl^a_n}(q^{1/2})$. 
The first four polynomials $h_n(q)$ are as follows:
\begin{gather*}
h_1(q)=1, \qquad h_2(q)=1+q,\\
h_3(q)=1+ 2q+ 3q^2+q^3,\\
h_4(q)=1+3q+7q^2+10q^3+10q^4+6q^{5}+q^{6}.
\end{gather*}
In general, the degree of $h_n(q)$ is equal to $n(n-1)/2$. Obviously, $h_n(1)=h_n$.

\subsection{Explicit formula}
Let $m\ge n\ge 0$. Then the $q$-binomial (Gaussian) coefficient $\qbinom{m}{n}$ is defined as
\[
\qbinom{m}{n}=\frac{m_q!}{n_q!(m-n)_q!},\ m_q!=\prod_{i=1}^m\frac{1-q^i}{1-q}.
\]
The following formula is obtained in \cite{CFR} using the geometry of quiver Grassmannians.
\begin{prop}
The Poincar\'e polynomial of the degenerate flag variety $\Fl^a_{n}$ is equal to
\begin{equation}\label{qG}
\sum_{f_1,\dots,f_{n-1}\ge 0}q^{\sum_{k=1}^{n-1} (k-f_k)(1-f_k+f_{k+1})}
\prod_{k=1}^{n-1} \qbinom{1+f_{k-1}}{f_k}
\prod_{k=1}^{n-1} \qbinom{1+f_{k+1}}{f_k},
\end{equation}
(we assume $f_0=f_n=0$).
\end{prop}

Such kind of formulas are usually referred to as fermionic: these are sums  of products of
$q$-binomial coefficients multiplied by certain powers of $q$. Geometrically, formula \eqref{qG} appears as
follows. The varieties $\Fl^a_n$ can be cut into disjoint pieces, such that each piece is fibered
over a product of several Grassmannians with fibers being affine spaces. Since the Poincar\'e polynomial
of a Grassmannian is given by a $q$-binomial coefficient, we arrive at the formula as above.

\section{Generating function and continued fraction}\label{cf} 
Our goal in this section is to give an explicit continued fraction form of the generating function
of the Poincar\'e polynomials $h_n(q)$ and to prove that they coincide with the $q$-versions
of the normalized median Genocchi numbers defined in \cite{HZ1}, \cite{HZ2}.

We note that formula \eqref{qG} can be seen as a sum over the set $M_{n}$ of Motzkin paths 
$\bff=(f_0,f_1,\dots,f_n)$ starting at $(0,f_0)=(0,0)$ and ending at $(n,f_n)=(n,0)$.

We first recall the formalism of the weighted generating functions of Motzkin paths due to Flajolet 
(see \cite{Fl}).
Let $\al_n$, $\beta_n$ and $\gamma_n$, $n\ge 0$  be
sequences of complex numbers called weights. For a nonnegative integer $k$ we define 
$w(k,k)=\gamma_k$, $w(k,k+1)=\al_k$ and $w(k,k-1)=\beta_{k-1}$ (if $k\ge 1$). 
We denote by $\al_\bullet$ the whole collection $(\al_k)_{k=0}^\infty$ and similarly for 
$\beta_\bullet$ and $\gamma_\bullet$. The weighted generating function 
of Motzkin paths is given by the formula
\[
F(s;\al_\bullet,\beta_\bullet,\gamma_\bullet)=\sum_{n\ge 0} s^n \sum_{\bff\in M_n} \prod_{k=0}^{n-1} w(f_k,f_{k+1}).
\] 
The following result is due to Flajolet \cite{Fl}.
\begin{thm}\label{Fl}
The weighted generating sum of the Motzkin paths is given by the continued fraction
\[
F(s;\al_\bullet,\beta_\bullet,\gamma_\bullet)=
\cfrac{1}{1-\gamma_0s-\cfrac{\al_0\beta_0s^2}{1-\gamma_1s-\cfrac{\al_1\beta_1s^2}{1-\gamma_2s-\dots}}}.
\]
\end{thm}

Let us apply this formalism to our situation. 
Formula \eqref{qG} can be rewritten as follows.
\begin{equation}\label{n(n-1)}
h_n(q)=q^{n(n-1)/2} \sum_{\bff\in M_n} q^{\sum_{k=1}^{n-1} f_k(f_k-f_{k+1}-2)} 
\prod_{k=1}^{n-1} \qbinom{1+f_{k-1}}{f_k}
\prod_{k=1}^{n-1} \qbinom{1+f_{k+1}}{f_k}.  
\end{equation}

We introduce three sequences of weights
\begin{equation*}
\al_m(q)=q^{-3m}\qbinom{m+2}{2},\ \beta_m(q)=q^{-m-1}\qbinom{m+2}{2},\ \gamma_m(q)=q^{-2m}\qbinom{m+1}{1}^2
\end{equation*}
and define $w(f_k,f_{k+1})$ using these weights. 
Then formula \eqref{n(n-1)} implies the following lemma.
\begin{lem}\label{L-1} 
\begin{equation}\label{-1}
q^{-n(n-1)/2}h_n(q)=\sum_{\bff\in M_n} \prod_{k=0}^{n-1} w(f_k,f_{k+1}).
\end{equation}
\end{lem}
In order to use the Flajolet theorem we need to get rid of the factor $q^{n(n-1)/2}$ in \eqref{-1}.
We introduce the notation 
$$\tilde h_n(q)=q^{n(n-1)/2} h_n(q^{-1})$$ 
(note that the degree of $h_n(q)$ is exactly $n(n-1)/2$). 
Let $\tilde h(q,s)=\sum_{n\ge 0} \tilde h_n(q) s^n$. We note that
\[
\gamma_m(q)=\binom{m+1}{1}_{\! {q^{-1}}}^2,\ \al_m(q)\beta_m(q)=q^{-1}\binom{m+2}{2}^2_{\! {q^{-1}}}.
\] 
Using Theorem \ref{Fl} we arrive at the following theorem.
\begin{thm}\label{main}
The generating function  $\tilde h(q,s)$ can be written as follows
\begin{equation}\label{f1}
\tilde h(q,s) = \cfrac{1}{1-s
          - \cfrac{qs^2}{1-\qbinom{2}{1}^2s
          - \cfrac{q\qbinom{3}{2}^2s^2}{1-\qbinom{3}{1}^2s
          - \cfrac{q\qbinom{4}{2}^2s^2}{1-\qbinom{4}{1}^2s-\dots}}}}
\end{equation}
\end{thm}
\begin{proof}
Follows from Lemma \ref{L-1} and the Flajolet theorem. 
\end{proof}

\begin{cor}
\begin{equation}\label{f2}
\tilde h(q,s) = \cfrac{1}{1
          - \cfrac{s}{1
          - \cfrac{qs}{1
          - \cfrac{\qbinom{3}{2}s}{1
          -\cfrac{q\qbinom{3}{2}s}{1
          - \cfrac{\qbinom{4}{2}s}{1
          -\cfrac{q\qbinom{4}{2}s}{1
          -\dots}}}}}}}
\end{equation}
\end{cor}
\begin{proof}
Recall the following formula (see \cite{DZ}, Lemma $2$)
\[
\cfrac{c_0}{1-c_1s-\cfrac{c_1c_2s^2}{1-(c_2+c_3)s-\cfrac{c_3c_4s^2}{1-(c_4+c_5)s-\dots}}}=
\cfrac{c_0}{1-\cfrac{c_1s}{1-\cfrac{c_2s}{1-\cfrac{c_3s}{1-\dots}}}}.
\]  
Now our corollary follows from Theorem \ref{main}.
\end{proof}

Specializing at $q=1$ we arrive at formula \eqref{Mp} for the normalized median Genocchi numbers.
\begin{cor}
The number $h_n$ is equal to the weighted sum over the set $M_n$ of  Motzkin paths 
$\sum_{\bff\in M_n} \prod_{k=0}^{n-1} w(f_k,f_{k+1})$
with the weights $w(\cdot,\cdot)$ defined by 
\[
\al_m=(m+1)(m+2)/2=\beta_m,\ \gamma_m=(m+1)^2.
\]
The generating function $\sum_{n\ge 0} h_n s^n$ is given by the continued fraction
\begin{equation}\label{hn}
\cfrac{1}{1
          - \cfrac{s}{1
          - \cfrac{s}{1
          - \cfrac{3s}{1
          -\cfrac{3s}{1
          - \cfrac{6s}{1
          -\cfrac{6s}{1
          -\cfrac{10s}{1
          -\dots}}}}}}}}
\end{equation}
\end{cor}

\begin{cor}
$\tilde h_n(q)=\bar c_{n+1}(q)$.
\end{cor}
\begin{proof}
Formula $(18)$ in \cite{HZ1} gives a continued fraction form of the generating function of the polynomials
$\bar c_n(q)$. Comparing this formula with \eqref{f2} and using the equations
\begin{gather*}
\qbinom{2n}{2}=(1+q^2+q^4+\dots+ q^{2n-2})(1+q+q^2+\dots+q^{2n-2}),\\
\qbinom{2n+1}{2}=(1+q^2+q^4+\dots+ q^{2n-2})(1+q+q^2+\dots+q^{2n}),
\end{gather*}
we obtain $\tilde h_n(q)=\bar c_{n+1}(q)$.
\end{proof}

We also derive the Viennot formula for the generating function of the (non-normalized) median Genocchi numbers $H_{2n+1}$ (see \cite{Du}, \cite{DZ}, \cite{V1}, \cite{V3}):
\begin{cor}
\[
1+\sum_{n=1}^\infty H_{2n-1}s^n=
\cfrac{1}{1-\cfrac{1^2s}{1-\cfrac{1^2s}{1-\cfrac{2^2s}{1-\cfrac{2^2s}{\dots}}}}}
\]
\end{cor}
\begin{proof}
Recall the formula $H_{2n+1}=2^{n}h_n$. Therefore, one has
\[
1+\sum_{n=1}^\infty H_{2n-1}s^n=1+\sum_{n=0}^\infty h_n2^n s^{n+1}=
1+s\sum_{n=0}^\infty h_n (2s)^n.
\]
Specializing  \eqref{f1} at $q=1$, we obtain that the generating function for the (non-normalized) median Genocchi numbers is given by
\begin{equation*}
1+\cfrac{s}{1-2s-
  \cfrac{4s^2}{1-2\cdot 4s-
  \cfrac{4\cdot 3^2s^2}{1-2\cdot 9s-
  \cfrac{4\cdot 6^2s^2}{1-2\cdot 16s-
  \cfrac{4\cdot 10^2s^2}{1-2\cdot 25s-\dots}}}}}.
\end{equation*}
Finally, we use the following formula from \cite{DZ}, Lemma 2:
\[
c_0+\cfrac{c_0c_1s}{1-(c_1+c_2)s-
    \cfrac{c_2c_3s^2}{1-(c_3+c_4)s-
    \cfrac{c_4c_5s^2}{1-\dots}}}=
\cfrac{c_0}{1-
\cfrac{c_1s}{1-
\cfrac{c_2s}{1-
\cfrac{c_3s}{1-\dots}}}}.
\]
\end{proof}

\section*{Acknowledgments}
We are grateful to Michael Finkelberg and Sergei Lando for useful discussions. 
This work was partially supported
by the RFBR Grant 09-01-00058,
by the grant Scientific Schools 6501.2010.2 and by the Dynasty Foundation.

\end{document}